\documentclass[11pt,reqno]{amsart}

\usepackage{geometry}

\title{A note on Hodge-Tate Spectral Sequences}
\author{Zhiyou Wu}

\usepackage{amsmath,amssymb,fancyhdr,url,amsthm,mathtools,mathrsfs,eufrak,setspace}
\usepackage[pdftex]{graphicx}
\usepackage{tikz-cd}
\usetikzlibrary{cd}
\usepackage[mathscr]{euscript}
\usepackage{tikz}
\usepackage{dsfont}
\usetikzlibrary{matrix}

\makeatletter
\newcommand*{\rom}[1]{\expandafter\@slowromancap\romannumeral #1@}
\makeatother

\usepackage{relsize} 
\usepackage[bbgreekl]{mathbbol} 
\usepackage{amsfonts} 
\DeclareSymbolFontAlphabet{\mathbb}{AMSb} 
\DeclareSymbolFontAlphabet{\mathbbl}{bbold} 
\newcommand{\prism}{{\mathlarger{\mathbbl{\Delta}}}}

\usepackage[toc,page]{appendix}



\newtheorem{proposition}{Proposition}[section]
\newtheorem{theorem}[proposition]{Theorem}

\newtheorem{corollary}[proposition]{Corollary}
\newtheorem{definition}[proposition]{Definition}
\newtheorem{remark}[proposition]{Remark}
\newtheorem{lemma}[proposition]{Lemma}

\tikzset{
    labl/.style={anchor=north, rotate=90, inner sep=.8mm}
}

\newenvironment{myproof}[1][\proofname]{%
  \begin{proof}[#1]\par\nobreak\ignorespaces
}{%
  \end{proof}
}

\begin{document}

\maketitle

\begin{abstract}
We prove that the Hodge-Tate spectral sequence of a proper smooth rigid analytic variety can be reconstructed from its infinitesimal $\mathbb{B}_{\text{dR}}^+$-cohomology through the Bialynicki-Birula map. A refinement of the decalage functor $L\eta$ is introduced to accomplish the proof.  Further, we give a new proof of the torsion-freeness of the infinitesimal $\mathbb{B}_{\text{dR}}^+$-cohomology
independent of Conrad-Gabber spreading theorem, and a conceptual explanation that the degeneration of Hodge-Tate spectral sequences is equivalent to that of Hodge-de Rham spectral sequences. 
\end{abstract}

\section{Introduction}

Let $X$ be a proper smooth rigid analytic variety over a complete  algebraically closed non-archimedean field $\mathbb{C}$
of mixed characteristic $p$. It is well-known that there is a Hodge-Tate spectral sequence
\[
E_2^{p,q} = H^{p}(X,\Omega_{X}^q)(-q) \ \ \Longrightarrow H^{p+q}_{\text{ét}}(X,\mathbb{Q}_p) \otimes_{\mathbb{Q}_p} \mathbb{C}
\]
from Hodge cohomology groups converging to the $p$-adic étale cohomology. The Hodge-Tate spectral sequences play an important role in some of the recent breakthroughs in arithmetic geometry, namely it is used to define  Hodge-Tate period maps as in \cite{Scholze2015}, \cite{Caraiani_2017} and \cite{CS2}, which is applied to prove the modularity theorem, and many other important results in \cite{10author}.

The spectral sequence is constructed from the truncated filtration $\tau_{\leq q} R\nu_{*} \hat{\mathcal{O}}_X$ on $R\nu_{*} \hat{\mathcal{O}}_X$, where 
\[
\nu: X_{\text{proét}} \rightarrow X_{\text{ét}}
\]
is the structure map from the proétale site to the étale site of $X$, and $\hat{\mathcal{O}}_X$ 
is the $p$-adic completion of the structure sheaf $\nu^* \mathcal{O}_{X_{\text{ét}}}$. In other words, it is the Grothendieck spectral sequence associated to the composite derived functor $R\Gamma(X_{\text{ét}}, - ) \circ R\nu_{*}$ evaluated on
$\hat{\mathcal{O}}_X$.  
The crucial ingredients in the construction is that we have canonical isomorphisms 
$H^i(X_{\text{proét}}, \hat{\mathcal{O}}_X) \cong H^i(X_{\text{ét}}, \mathbb{Q}_p) \otimes_{\mathbb{Q}_p} \mathbb{C}$, and 
$R^q\nu_{*} \hat{\mathcal{O}}_X \cong \Omega^q_{X_{\text{ét}}}(-q)$, see \cite{scholze2012perfectoid} section 3.3. 

On the other hand, we know from \cite{Bhatt2018} section 13 (see also \cite{Guo1} for a site theoretic construction of it) that there is a $\mathbb{B}^+_{\text{dR}}$-cohomology theory
\[
H^i_{\text{crys}}(X/\mathbb{B}^+_{\text{dR}})
\]
associated to $X$, where $\mathbb{B}^+_{\text{dR}}$ is the $\xi$-adic completion of $W(\mathcal{O}_{\mathbb{C}^{\flat}})[\frac{1}{p}]$, with $\mathbb{C}^{\flat}$ the tilt of $\mathbb{C}$ and $\xi$ a generator of the kernel of the canonical morphism 
$W(\mathcal{O}_{\mathbb{C}^{\flat}}) \rightarrow \mathcal{O}_{\mathbb{C}}$. They are finite free $\mathbb{B}^+_{\text{dR}}$-modules equipped with canonical morphisms
\[
H^i_{\text{crys}}(X/\mathbb{B}^+_{\text{dR}})  \longrightarrow
H^{i}_{\text{ét}}(X,\mathbb{Q}_p) \otimes_{\mathbb{Q}_p} \mathbb{B}^+_{\text{dR}}
\]
which induces identifications
\[
H^i_{\text{crys}}(X/\mathbb{B}^+_{\text{dR}}) \otimes_{\mathbb{B}^+_{\text{dR}}} \mathbb{B}_{\text{dR}} \cong
H^{i}_{\text{ét}}(X,\mathbb{Q}_p) \otimes_{\mathbb{Q}_p} \mathbb{B}_{\text{dR}}
\]
where $\mathbb{B}_{\text{dR}} := \mathbb{B}_{\text{dR}}^+[\frac{1}{\xi}]$. 

We view the above identification as providing a $\mathbb{B}_{\text{dR}}^+$-lattice 
$H^i_{\text{crys}}(X/\mathbb{B}^+_{\text{dR}})$ 
in the $\mathbb{B}_{\text{dR}}$-vector space $H^{i}_{\text{ét}}(X,\mathbb{Q}_p) \otimes_{\mathbb{Q}_p} \mathbb{B}_{\text{dR}}$. Note that there is another natural $\mathbb{B}_{\text{dR}}^+$-lattice 
$
H^{i}_{\text{ét}}(X,\mathbb{Q}_p) \otimes_{\mathbb{Q}_p} \mathbb{B}_{\text{dR}}^+
$
inside it, and the relative position of the two lattices is measured by
\[
\text{Fil}_m := 
(H^i_{\text{crys}}(X/\mathbb{B}^+_{\text{dR}}) \cap 
\xi^m
(H^{i}_{\text{ét}}(X,\mathbb{Q}_p) \otimes_{\mathbb{Q}_p} \mathbb{B}_{\text{dR}}^+))/
(H^i_{\text{crys}}(X/\mathbb{B}^+_{\text{dR}}) \cap 
\xi^{m+1}
(H^{i}_{\text{ét}}(X,\mathbb{Q}_p) \otimes_{\mathbb{Q}_p} \mathbb{B}_{\text{dR}}^+))
\]
viewed as subspaces of
\[
\xi^{m}
(H^{i}_{\text{ét}}(X,\mathbb{Q}_p) \otimes_{\mathbb{Q}_p} \mathbb{B}_{\text{dR}}^+)/
\xi^{m+1}
(H^{i}_{\text{ét}}(X,\mathbb{Q}_p) \otimes_{\mathbb{Q}_p} \mathbb{B}_{\text{dR}}^+)
=
H^{i}_{\text{ét}}(X,\mathbb{Q}_p) \otimes_{\mathbb{Q}_p} \mathbb{C} (m).
\]
The main result of this paper confirms that this is the same filtration induced by Hodge-Tate spectral sequence on $H^{i}_{\text{ét}}(X,\mathbb{Q}_p) \otimes_{\mathbb{Q}_p} \mathbb{C}$. 

\begin{theorem}
The filtration
$\text{Fil}_m (-m)$ is the same as the filtration induced by Hodge-Tate spectral sequence on 
$H^{i}_{\text{ét}}(X,\mathbb{Q}_p) \otimes_{\mathbb{Q}_p} \mathbb{C}$. More precisely, $\text{Fil}_m(-m)$ is equal to the image of 
\[
H^i(X_{\text{ét}},\tau_{\leq m} R\nu_{*} \hat{\mathcal{O}}_X) \longrightarrow 
H^i(X_{\text{ét}}, R\nu_{*} \hat{\mathcal{O}}_X) 
\cong 
H^i(X_{\text{proét}}, \hat{\mathcal{O}}_X) 
\cong 
H^{i}_{\text{ét}}(X,\mathbb{Q}_p) \otimes_{\mathbb{Q}_p} \mathbb{C},
\]
and we have a canonical identification
\[
\text{Fil}_m(-m) /\text{Fil}_{m+1}(-1-m)  \cong H^i(X, \Omega_X^m) (-m).
\]
\end{theorem}

The result is well-known among experts. For example, the interpretation in terms of the Bialynicki-Birula map and the theory in \cite{SCHOLZE2013} (to be explained in the following remarks)  is used in one of the constructions of Hodge-Tate period maps as in \cite{Caraiani_2017} and \cite{hansen2016period}. However, to the best of the author's knowledge, the formulation as in the theorem in terms of 
$H^i_{\text{crys}}(X/\mathbb{B}^+_{\text{dR}})$
is not explicit in the literature. It gives a new construction of the Hodge-Tate filtration 
in terms of the $\mathbb{B}^+_{\text{dR}}$-cohomology 
$H^i_{\text{crys}}(X/\mathbb{B}^+_{\text{dR}})$.
\footnote{The $\mathbb{B}^+_{\text{dR}}$-cohomology is new only for rigid analytic varieties defined over $\mathbb{C}$. If the space is defined over a discretely valued field, it  essentially reduced to the theory in \cite{SCHOLZE2013}. } 

The proof is to study a refinement of the decalage functor $L\eta$ introduced in \cite{Bhatt2018}. It is a functor
\[
L\eta_{\mathcal{I},\bullet}: D(\mathcal{O}_T) \longrightarrow FD(\mathcal{O}_T)
\]
from the derived category of a ring topos $T$ to its filtered derived category. Its graded pieces can be  identified, so does its reduction mod $\mathcal{I}$. 

As a byproduct of our treatment of the Hodge-Tate filtration, we give a new proof of the torsion-freeness of $H^i_{\text{crys}}(X/\mathbb{B}^+_{\text{dR}})$
that is independent of Conrad-Gabber. 
It is complementary to the proofs found in the literature (see \cite{Bhatt2018} 13.19 and \cite{Guo1} theorem 7.3.5), which to the best of the author's knowledge all make use of Conrad-Gabber.  

Moreover, we provide a conceptual explanation of the equivalence between degeneration of Hodge-de Rham and Hodge-Tate spectral sequence, which 
is implicit in the proof of \cite{Bhatt2018} theorem 13.3.

\begin{remark} \label{hthdsp}
The statement of the theorem includes that the Hodge-Tate spectral sequence degenerates at $E_2$. This is actually used as an input in the proof. More precisely, we use the degeneration of both Hodge-de Rham and Hodge-Tate spectral sequences
\footnote{Indeed, we provide two proofs, one of which only uses the degeneration of Hodge-Tate spectral sequences.}, which is proved in \cite{Bhatt2018} theorem 13.3, see also \cite{Guo1} theorem 7.3.5. and \cite{Guo2022-vo} theorem 1.1.3. Both approaches depend on Conrad-Gabber spreading theorem. It would be interesting if we can find a direct proof of the degeneration. 
\end{remark}

\begin{remark}
The $\mathbb{B}_{\text{dR}}^+$-lattices in $H^{i}_{\text{ét}}(X,\mathbb{Q}_p) \otimes_{\mathbb{Q}_p} \mathbb{B}_{\text{dR}}$
are parameterized by the $\mathbb{C}$-points of $\mathbb{B}_{\text{dR}}^+$-affine Grassmannian $Gr_{GL_n}^{\mathbb{B}_{\text{dR}}^+}$ as defined in \cite{Caraiani_2017} definition 3.4.1, see \cite{SW17} lecture 19 as well, where $n$ is the dimension of $H^{i}_{\text{ét}}(X,\mathbb{Q}_p)$. The subspaces $Fil_m(-m)$  are parameterized by the ($\mathbb{C}$-points of) flag variety with respect to the vector space
$H^{i}_{\text{ét}}(X,\mathbb{Q}_p) \otimes_{\mathbb{Q}_p} \mathbb{C}$, and $Fil_m(-m)$ associated to 
the lattice
$H^i_{\text{crys}}(X/\mathbb{B}^+_{\text{dR}})$
is exactly its image under the Bialynicki-Birula map as defined in \cite{Caraiani_2017} proposition 3.4.3 from the affine Grassmannian
to the flag variety. \footnote{To be fully precise, the Bialynicki-Birula map is only defined on an open Schubert cell of the affine Grassmannian, so we need to first fix the type of the relative position between lattices (which also specifies the flag variety) to actually have the Bialynicki-Birula map, but this creates no problem. }
\end{remark}

\begin{remark}
When $X$ is defined over a discretely valued field, we know from the results in \cite{Bhatt2018} section 13 that 
$H^i_{\text{crys}}(X_{\mathbb{C}}/\mathbb{B}^+_{\text{dR}})$
can be computed as the de Rham cohomology of $X$ base changed to $\mathbb{B}^+_{\text{dR}}$. Then the theorem follows from \cite{SCHOLZE2013} theorem 8.4. 

Moreover, using the spreading theorem of Conrad-Gabber (see \cite{Bhatt2018} corollary 13.16), we can deduce the theorem from \cite{SCHOLZE2013} theorem 8.8 and proposition 7.9. Note that the condition of theorem 8.8. in $loc.cit.$ is proved in \cite{SW17} theorem 10.5.1. We believe that it is more natural to prove the theorem directly from the very construction of 
$H^i_{\text{crys}}(X/\mathbb{B}^+_{\text{dR}})$,
thereby avoiding the use of  theory of $\mathcal{O}\mathbb{B}^+_{\text{dR}}$-modules in \cite{SCHOLZE2013}. However, we still cannot avoid Conrad-Gabber in our treatment, see remark \ref{hthdsp}. 
\end{remark}

\subsection*{Acknowledgement}
I would like to thank Yu Min and Yupeng Wang for discussions related to this work. I would like to thank Peter Scholze for corrections on the draft.

\section{Recollections} 

Let us first recall the construction of 
$H^i_{\text{crys}}(X/\mathbb{B}^+_{\text{dR}})$. The original construction in \cite{Bhatt2018} is to construct an explicit complex for each small affinoid $X$, and then glue. In \cite{Guo1}, Guo defined an infinitesimal site and reconstruct 
$H^i_{\text{crys}}(X/\mathbb{B}^+_{\text{dR}})$
as the cohomology of the structure sheaf on this site. Further, the $\mathbb{B}^+_{\text{dR}}$-prismatic site is introduced in \cite{Guo2}, which unifies the previous constructions. We will use the formulation of $\mathbb{B}^+_{\text{dR}}$-prismatic site for convenience, but the original construction of \cite{Bhatt2018} can also be invoked here. 

For our purpose, we only need to know that 
$H^i_{\text{crys}}(X/\mathbb{B}^+_{\text{dR}})$
can be computed as the cohomology of a sheaf of complexes 
\[
L\eta_{\xi} \prism_{X/\mathbb{B}^+_{\text{dR}}}
\]
on the étale site $X_{\text{ét}}$ of $X$, where 
$L\eta_{\xi}$ is the decalage functor with respect to $\xi$ as defined in \cite{Bhatt2018} section 6, 
and 
$\prism_{X/\mathbb{B}^+_{\text{dR}}}$
is the sheaf of complexes 
\footnote{in the $\infty$-catogorical sense.}
of $\mathbb{B}^+_{\text{dR}}$-modules on $X_{\text{ét}}$ which sends every affinoid to its derived $\mathbb{B}^+_{\text{dR}}$-prismatic  cohomology (with respect to the structure sheaf) as defined in \cite{Guo2} section 2. In other words, we have 
\begin{equation} \label{defn}
H^i_{\text{crys}}(X/\mathbb{B}^+_{\text{dR}}) 
\cong 
R^i\Gamma(X_{\text{ét}},
L\eta_{\xi} \prism_{X/\mathbb{B}^+_{\text{dR}}}),
\end{equation}
which is \cite{Guo2} theorem 6.0.1.

It is proved in \cite{Guo2} theorem 5.1.1 and the discussion before it that we have a canonical quasi-isomorphism
\begin{equation} \label{derham}
L\eta_{\xi} \prism_{X/\mathbb{B}^+_{\text{dR}}} \otimes^L_{\mathbb{B}^+_{\text{dR}}} \mathbb{B}^+_{\text{dR}}/\xi
\cong 
H^{\bullet} (\prism_{X/\mathbb{B}^+_{\text{dR}}} /\xi) 
\cong \Omega_X^{\bullet},
\end{equation}
where the first isomorphism is \cite{Bhatt2018} proposition 6.12, and $\Omega_X^{\bullet}$ is the de Rham complex of $X$.

Moreover,  there is a natural quasi-isomorphism
\begin{equation} \label{proetale}
R\Gamma(X_{\text{ét}}, \prism_{X/\mathbb{B}^+_{\text{dR}}}) 
\cong 
R\Gamma(X_{\text{proét}}, \mathbb{B}^+_{\text{dR}})
\cong 
R\Gamma(X_{\text{ét}}, \mathbb{Q}_p) \otimes_{\mathbb{Q}_p} \mathbb{B}^+_{\text{dR}}
\end{equation}
where $\mathbb{B}^+_{\text{dR}}$ is, by abuse of notation, the $\mathbb{B}^+_{\text{dR}}$-period sheaf on $X_{\text{proét}}$ 
as defined in \cite{SCHOLZE2013} definition 6.1. The first isomorphism is \cite{Guo2} theorem 7.2.1, and the second is in the proof of \cite{SCHOLZE2013} theorem 8.4.  The canonical comparison morphism
\[
H^i_{\text{crys}}(X/\mathbb{B}^+_{\text{dR}})  \longrightarrow
H^{i}_{\text{ét}}(X,\mathbb{Q}_p) \otimes_{\mathbb{Q}_p} \mathbb{B}_{\text{dR}}^+
\]
is induced from the canonical map
\[
\iota: 
L\eta_{\xi} \prism_{X/\mathbb{B}^+_{\text{dR}}} 
\longrightarrow 
\prism_{X/\mathbb{B}^+_{\text{dR}}} 
\]
together with the identification. Note that $\iota$ exists because $\prism_{X/\mathbb{B}^+_{\text{dR}}}  \in D_{\geq 0}$
and $H^0(\prism_{X/\mathbb{B}^+_{\text{dR}}} )$
is $\xi$-torsion-free. 

Lastly, it follows from the proof of \cite{Guo2} theorem 7.2.1 that we have a natural quasi-isomorphism
\begin{equation} \label{reduction}
\overline{\prism}_{X/\mathbb{B}^+_{\text{dR}}} := \prism_{X/\mathbb{B}^+_{\text{dR}}} \otimes^L_{\mathbb{B}^+_{\text{dR}}} \mathbb{B}^+_{\text{dR}}/\xi
\cong 
R\nu_{*} \hat{\mathcal{O}}_X
\end{equation}
where 
$
\nu: X_{\text{proét}} \rightarrow X_{\text{ét}}
$
and $\hat{\mathcal{O}}_X$ are as in the introduction.

\section{A refinement of $L\eta$-functor}

In this section, we construct a refinement of the decalage functor $L\eta$ introduced in \cite{Bhatt2018}. We work in the same setting as  $loc.cit.$ 

Let $(T, \mathcal{O}_T)$
be a ringed topos, and $D(\mathcal{O}_T)$ the derived category of $\mathcal{O}_T$-modules. Let $\mathcal{I} \subset \mathcal{O}_T$ be an invertible ideal sheaf. 

Recall that a complex $K^{\bullet}$ of $\mathcal{O}_T$-modules is said to be $\mathcal{I}$-torsion-free if the canonical map 
$\mathcal{I}\otimes K^i \rightarrow K^i$ is injective for every $i$, and we can define a complex $(\eta_{\mathcal{I}} K)^{\bullet}$ with terms  
\[
(\eta_{\mathcal{I}} K )^i= \{ x \in K^i | dx \in \mathcal{I} K^{i+1} \} \otimes_{\mathcal{O}_T} \mathcal{I}^{\otimes i} 
\]
where $K^{\bullet}$ is a $\mathcal{I}$-torsion-free complex, and there is a natural differential map making $\eta_{\mathcal{I}} K^{\bullet} $ a complex. 
By \cite{Bhatt2018} lemma 6.4, we have 
\[
H^i(\eta_{\mathcal{I}} K^{\bullet}) \cong (H^i(K^{\bullet})/ H^i(K^{\bullet})[\mathcal{I}]) \otimes_{\mathcal{O}_T} \mathcal{I}^{\otimes i}, 
\]
so we can derive the construction to obtain a functor
\[
L\eta_{\mathcal{I}}: D(\mathcal{O}_T) \longrightarrow D(\mathcal{O}_T).
\]

\begin{definition}
Let $m \in \mathbb{Z}$, and $K^{\bullet }$ a $\mathcal{I}$-torsion-free complex of $\mathcal{O}_T$-modules. Define a new complex $(\eta_{\mathcal{I},m} K)^{\bullet}$ 
with terms
\[
(\eta_{\mathcal{I},m} K)^i = 
\begin{cases}
      (\eta_{\mathcal{I}}K)^i & i \geq m \\
      K^i \otimes_{\mathcal{O}_T} \mathcal{I}^{\otimes m} & i < m.
    \end{cases}  
\]
The differential  is inherited from that of $\eta_{\mathcal{I}}K$ (resp. $K^{\bullet}$ ($\otimes \text{Id}_{\mathcal{I}^{\otimes m}}$)) for $i \geq m$ (resp. $i< m-1$). For $i=m-1$, the differential
\[
d: K^{m-1} \otimes_{\mathcal{O}_T} \mathcal{I}^{\otimes m} \longrightarrow 
(\eta_{\mathcal{I}}K)^m = \{ x \in K^m | dx \in \mathcal{I} K^{m+1} \} \otimes_{\mathcal{O}_T} \mathcal{I}^{\otimes m} 
\]
is defined to be $d \otimes \text{Id}_{\mathcal{I}^{\otimes m}}$. 
\end{definition}

We first compute the cohomology of $\eta_{\mathcal{I},m} K^{\bullet}$. 

\begin{lemma}
Let $K^{\bullet}$ be an $\mathcal{I}$-torsion-free complex, 
then we have a natural isomorphism
\[
H^i(\eta_{\mathcal{I},m} K^{\bullet}) \cong
\begin{cases}
      H^i(\eta_{\mathcal{I}} K^{\bullet})  & i > m \\
      H^i( K^{\bullet}) \otimes_{\mathcal{O}_T} \mathcal{I}^{\otimes m} & i \leq m.
    \end{cases}  
\]
\end{lemma}

\begin{proof}
This is obvious except for $i=m,m-1$. For $i=m-1$, we simply note that $(\eta_{\mathcal{I}}K)^m$ is a subspace of $K^m \otimes_{\mathcal{O}_T} \mathcal{I}^{\otimes m}$, so the kernel of $d$ at $m-1$ is the same as  
$Z^{m-1}(K^{\bullet}) \otimes_{\mathcal{O}_T} \mathcal{I}^{\otimes m} $, 
the cocycles of $K^{\bullet}$ twisted by $\mathcal{I}^{\otimes m} $. 

For $i=m$, we observe that the cocycle space is 
$Z^m(K^{\bullet}) \otimes_{\mathcal{O}_T} \mathcal{I}^{\otimes m} $, and the boundary is exactly $B^m(K^{\bullet}) \otimes_{\mathcal{O}_T} \mathcal{I}^{\otimes m}$. 
\end{proof}

We can now dervie the $\eta_{\mathcal{I},m}$-construction to obtain a functor
\[
L\eta_{\mathcal{I},m}: D(\mathcal{O}_T) \longrightarrow D(\mathcal{O}_T).
\]

The most important property of 
$L\eta_{\mathcal{I},m}$ is that it forms a filtration. In other words, for every  $\mathcal{I}$-torsion-free complex  $K^{\bullet}$, we have a canonical filtration
\[
\cdots
\subset
\eta_{\mathcal{I},m+1} K^{\bullet} \subset \eta_{\mathcal{I},m} K^{\bullet} 
\subset \cdots.
\]
The inclusion at degrees $i\geq m+1$ (resp. $i\leq m-1$) is the identity of $(\eta_{\mathcal{I}} K^{\bullet})^i$ (resp. $\text{id}_{K^i \otimes_{\mathcal{O}_T} \mathcal{I}^{\otimes m}} \otimes (\mathcal{I} \hookrightarrow \mathcal{O}_T) $). For degree $m$, it is given by the canonical inclusion 
\[
(K^{m} \otimes 
\mathcal{I} )\otimes_{\mathcal{O}_T} \mathcal{I}^{\otimes m} \longrightarrow 
(\eta_{\mathcal{I}}K)^m = \{ x \in K^m | dx \in \mathcal{I} K^{m+1} \} \otimes_{\mathcal{O}_T} \mathcal{I}^{\otimes m}
\]
induced by $(\mathcal{I} \hookrightarrow \mathcal{O}_T )\otimes \text{id}_{\mathcal{I}^{\otimes m}}$. 
We can summarize the information as in the diagram 
\[
\begin{tikzcd}
\vdots \arrow[r,equal]
&
\vdots 
\\
(\eta_{\mathcal{I}}K)^{m+1} 
\arrow[r,equal] 
\arrow[u]
 &  
 (\eta_{\mathcal{I}}K)^{m+1}
 \arrow[u]
 \\
K^m \otimes_{\mathcal{O}_T} \mathcal{I}^{\otimes m+1}
\arrow[r, hook] 
\arrow[u]
 & 
 (\eta_{\mathcal{I}}K)^{m}
 \arrow[u]
 \\
 K^{m-1} \otimes_{\mathcal{O}_T} \mathcal{I}^{\otimes m+1}
\arrow[r, hook] 
\arrow[u]
&
K^{m-1} \otimes_{\mathcal{O}_T} \mathcal{I}^{\otimes m}
\arrow[u]
\\
\vdots \arrow[r,hook] \arrow[u]
&
\vdots \arrow[u]
\end{tikzcd}
\]

We can compute the graded pieces of the filtration.

\begin{lemma} \label{graded}
Let $K^{\bullet}$ be an $\mathcal{I}$-torsion-free complex, we have a canonical isomorphism
\[
\eta_{\mathcal{I},m} K^{\bullet}/ \eta_{\mathcal{I},m+1} K^{\bullet} 
\cong 
\tau_{\leq m} (K^{\bullet} \otimes_{\mathcal{O}_T} \mathcal{O}_T/\mathcal{I}) \otimes_{\mathcal{O}_T} \mathcal{I}^{\otimes m}.
\]
\end{lemma}

\begin{proof}
This is obvious except for degree $m$. We compute that
\[
(\eta_{\mathcal{I}}K)^{m}/
(K^m \otimes_{\mathcal{O}_T} \mathcal{I}^{\otimes m+1} )
\cong 
\frac{
\{ x \in K^m | dx \in \mathcal{I} K^{m+1} \}}{\mathcal{I} K^m}
\otimes_{\mathcal{O}_T}  \mathcal{I}^{\otimes m}
\]
\[
\cong 
\{ \Bar{x} \in K^m \otimes_{\mathcal{O}_T} \mathcal{O}_T/\mathcal{I} \ | \ d\Bar{x} = 0 \in  K^{m+1} \otimes_{\mathcal{O}_T} \mathcal{O}_T/\mathcal{I}\}
\otimes_{\mathcal{O}_T}  \mathcal{I}^{\otimes m},
\]
but this is exactly the degree $m$ part of
$\tau_{\leq m} (K^{\bullet} \otimes_{\mathcal{O}_T} \mathcal{O}_T/\mathcal{I}) \otimes_{\mathcal{O}_T} \mathcal{I}^{\otimes m}$
by definition of $\tau_{\leq m} $. 
\end{proof}

A good way to package the information of 
$\eta_{\mathcal{I},m}$-construction  is to view 
$\{ L \eta_{\mathcal{I},m} \}_{m \in \mathbb{Z}}$ 
as a functor
\[
L \eta_{\mathcal{I},\bullet} : D(\mathcal{O}_T) 
\longrightarrow 
FD(\mathcal{O}_T)
\]
from the derived category 
$D(\mathcal{O}_T)$ 
to the filtered derived category of $\mathcal{O}_T$. Moreover, the graded quotient functor can be identified as 
\[
\text{Gr}^m(L \eta_{\mathcal{I},\bullet} (-)) \cong 
\tau_{\leq m} (- \otimes^L_{\mathcal{O}_T} \mathcal{O}_T/\mathcal{I}) \otimes_{\mathcal{O}_T} \mathcal{I}^{\otimes m}. 
\]

The next result we want to establish is the behaviour of $L \eta_{\mathcal{I},\bullet}$ under the quotient by $\mathcal{I}$. Recall as in \cite{Bhatt2018} proposition 6.12 we have a canonical quasi-isomorphism
\[
(L \eta_{\mathcal{I}} K )\otimes^L_{\mathcal{O}_T} \mathcal{O}_T/\mathcal{I}  \cong H^{\bullet}( K/\mathcal{I})
\]
where 
$H^{\bullet}( K/\mathcal{I})$
is the complex whose degree $i$-th term is $H^i(K \otimes^L_{\mathcal{O}_T} \mathcal{O}_T/\mathcal{I}) \otimes_{\mathcal{O}_T} \mathcal{I}^{\otimes i}$
and the differential is given by the Bockstein map with respect to
\[
0 \longrightarrow \mathcal{I}/\mathcal{I}^2 \longrightarrow \mathcal{O}_T/\mathcal{I}^2 \longrightarrow \mathcal{O}_T/\mathcal{I} \longrightarrow 0.  
\]
We will prove a slight refinement of the result for $L \eta_{\mathcal{I},\bullet}$. 

Let $K^{\bullet}$ be a $\mathcal{I}$-torsion-free complex, we first observe that there is a canonical filtration
\[
\eta_{\mathcal{I},m} K^{\bullet} \otimes_{\mathcal{O}_T} \mathcal{I}  \subset \eta_{\mathcal{I},m+1} K^{\bullet} 
\subset
\eta_{\mathcal{I},m} K^{\bullet},
\]
which in particular explains why $\eta_{\mathcal{I},m} K^{\bullet}/\eta_{\mathcal{I},m+1} K^{\bullet}$
is an $\mathcal{O}_T / \mathcal{I}$-complex. We have identified the second graded piece, and the first graded piece can also be computed as follows. 

\begin{proposition} \label{modI}
Let $K^{\bullet}$ be a $\mathcal{I}$-torsion-free complex, then we have a canonical quasi-isomorphism
\[
\eta_{\mathcal{I},m+1} K^{\bullet}
/
\eta_{\mathcal{I},m} K^{\bullet} \otimes_{\mathcal{O}_T} \mathcal{I} 
\cong 
F_{m+1} H^{\bullet}( K/\mathcal{I})
\]
where $F_{m+1}$ is the Hodge filtration of complexes, i.e. 
\[
(F_{m+1} C^{\bullet})^i =
\begin{cases}
      C^i  & i \geq m+1 \\
      0 & i \leq m.
    \end{cases}  
\]
\end{proposition}

\begin{proof}
We see immediately from the definition that the degree $\geq m+1$ (resp. $\leq m-1$) part of the left hand side is $(\eta_{\mathcal{I}} K^{\bullet})^i \otimes_{\mathcal{O}_T} \mathcal{O}_T/\mathcal{I}$ (resp. 0).  
It remains to identify the degree $m$-part of the LHS. 

By definition 
\[
(\eta_{\mathcal{I},m+1} K^{\bullet}
/
\eta_{\mathcal{I},m} K^{\bullet} \otimes_{\mathcal{O}_T} \mathcal{I})^m =
\frac{K^m }{ \{ x \in K^m | dx \in \mathcal{I} K^{m+1} \}  } 
\otimes_{\mathcal{O}_T} \mathcal{I}^{\otimes m+1},
\]
which is identified through the differential with 
\[
B^{m+1}(K^{\bullet} \otimes_{\mathcal{O}_T} \mathcal{O}_T / \mathcal{I}) \otimes_{\mathcal{O}_T} \mathcal{I}^{\otimes m+1}
\]
where $B^{m+1}(K^{\bullet} \otimes_{\mathcal{O}_T} \mathcal{O}_T / \mathcal{I})$
is the image of $(K^{\bullet} \otimes_{\mathcal{O}_T} \mathcal{O}_T / \mathcal{I})^m$ inside 
$(K^{\bullet} \otimes_{\mathcal{O}_T} \mathcal{O}_T / \mathcal{I})^{m+1}$. 

We first prove that the $m$-th cohomology of the LHS is trivial. This is equivalent to the differential map
\[
\frac{K^m }{ \{ x \in K^m | dx \in \mathcal{I} K^{m+1} \}  } 
\otimes_{\mathcal{O}_T} \mathcal{I}^{\otimes m+1} 
\longrightarrow
(\eta_{\mathcal{I}} K^{\bullet})^{m+1} \otimes_{\mathcal{O}_T} \mathcal{O}_T/\mathcal{I}
\]
being injective. But this is clear, if $x \in K^m$ is mapped to $\mathcal{I}  \{ y \in K^{m+1} | dy \in \mathcal{I} K^{m+2} \}  $, namely $dx \in \mathcal{I}  \{ y \in K^{m+1} | dy \in \mathcal{I} K^{m+2} \}  \subset \mathcal{I} K^{m+1}$, proving that $\Bar{x}$ is zero in the left hand side. 

We can now proceed exactly as in the proof of 
 \cite{Bhatt2018} proposition 6.12. We have canonical maps 
 \[
 (\eta_{\mathcal{I}} K^{\bullet})^{i} \otimes_{\mathcal{O}_T} \mathcal{O}_T/\mathcal{I} 
 \longrightarrow 
 Z^i(K^{\bullet} \otimes_{\mathcal{O}_T} \mathcal{O}_T/\mathcal{I}) \otimes_{\mathcal{O}_T} \mathcal{I}^{\otimes i} 
 \]
 which induces the quasi-isomorphism
\[
\eta_{\mathcal{I}} K^{\bullet} \otimes_{\mathcal{O}_T} \mathcal{O}_T/\mathcal{I} 
\cong 
H^{\bullet} (K/\mathcal{I}). 
\]
We have a commutative diagram
\[
\begin{tikzcd}
\vdots
\arrow[r]
&
\vdots
\\
(\eta_{\mathcal{I}} K)^{m+3}
\otimes_{\mathcal{O}_T} \mathcal{O}_T/\mathcal{I}
\arrow[r]
\arrow[u]
& 
H^{m+3}(K^{\bullet} \otimes_{\mathcal{O}_T} \mathcal{O}_T/\mathcal{I}) \otimes_{\mathcal{O}_T} \mathcal{I}^{\otimes m+3} 
\arrow[u]
\\
(\eta_{\mathcal{I}} K)^{m+2}
\otimes_{\mathcal{O}_T} \mathcal{O}_T/\mathcal{I}
\arrow[r]
\arrow[u]
& 
H^{m+2}(K^{\bullet} \otimes_{\mathcal{O}_T} \mathcal{O}_T/\mathcal{I}) \otimes_{\mathcal{O}_T} \mathcal{I}^{\otimes m+2} 
\arrow[u]
\\
(\eta_{\mathcal{I}} K)^{m+1}
\otimes_{\mathcal{O}_T} \mathcal{O}_T/\mathcal{I}
\arrow[u]
\arrow[r]
&
Z^{m+1}(K^{\bullet} \otimes_{\mathcal{O}_T} \mathcal{O}_T/\mathcal{I}) \otimes_{\mathcal{O}_T} \mathcal{I}^{\otimes m+1} 
\arrow[u]
\\
\frac{K^m }{ \{ x \in K^m | dx \in \mathcal{I} K^{m+1} \}  } 
\otimes_{\mathcal{O}_T} \mathcal{I}^{\otimes m+1}
\arrow[u]
\arrow[r, "d","\sim"']
&
B^{m+1}(K^{\bullet} \otimes_{\mathcal{O}_T} \mathcal{O}_T / \mathcal{I}) \otimes_{\mathcal{O}_T} \mathcal{I}^{\otimes m+1}
\arrow[u]
\\
0
\arrow[r]
\arrow[u]
&
0
\arrow[u]
\end{tikzcd}
\]
where the differential on right hand side is induced from the Bockstein map, except the lowest one, which is the canonical inclusion. 

It induces a quasi-isomorphism in degrees $\geq m+2$ by \cite{Bhatt2018} proposition 6.12, and we have seen that the degree $m$ cohomology of both sides are trivial. Thus it remains to show that it induces a quasi-isomorphism in degree $m+1$, this is the same as the proof of $loc.cit.$.  

We first prove that the map on $m+1$-th cohomology is surjective. Let $\Bar{x}
\in Z^{m+1}(K^{\bullet} \otimes_{\mathcal{O}_T} \mathcal{O}_T/\mathcal{I})$  
be a cocycle under Bockstein map, then 
by definition there exists a lift $x \in K^{m+1}$ of $\Bar{x}$ together with $y \in \mathcal{I} K^{m+1}$
such that $dx \equiv dy \ \text{mod} \ \mathcal{I}^2 K^{m+2}$. This implies that
$x \in  \{ x \in K^{m+1} | dx \in \mathcal{I} K^{m+2} \} $. Moreover, it tells us that 
$dx \in \mathcal{I} \{ x \in K^{m+2} | dx \in \mathcal{I} K^{m+3} \} $, which means $x$ defines a cocycle on the LHS that maps to $\bar{x}$. 

Now we show that the map on cohomology is injective. Let 
$x \in \{ x \in K^{m+1} | dx \in \mathcal{I} K^{m+2} \} $  whose reduction in 
$\{ x \in K^{m+1} | dx \in \mathcal{I} K^{m+2} \}/ \mathcal{I} \{ x \in K^{m+1} | dx \in \mathcal{I} K^{m+2} \}$
is a cocycle on the LHS. Moreover, we assume that its reduction $\bar{x} \in K^{m+1}/\mathcal{I} 
$ 
lies in 
$B^{m+1}(K^{\bullet} \otimes_{\mathcal{O}_T} \mathcal{O}_T / \mathcal{I})$. 
By the isomorphism of the lower horizontal map, there exists $y \in K^m$ such that 
$dy \equiv x \ \text{mod} \ \mathcal{I} K^{m+1}$. The cocycle condition implies that $dx \in \mathcal{I}^2 \{ x \in K^{m+2} | dx \in \mathcal{I} K^{m+3} \}$, but this  implies that 
$dy \equiv x \ \text{mod} \ \mathcal{I} \{ x \in K^{m+1} | dx \in \mathcal{I} K^{m+2} \}$, i.e. $x$ defines a trivial cohomological class on the LHS.  
\end{proof}

\begin{corollary} \label{corollary}
Let $K^{\bullet}$ be a $\mathcal{I}$-torsion-free complex, then
\[
H^i(\eta_{\mathcal{I},m} K^{\bullet} \otimes_{\mathcal{O}_T} \mathcal{O}_T /\mathcal{I} )
\cong 
\begin{cases}
      H^i(H^{\bullet}(K/\mathcal{I}))  
      & i \geq m+1 
      \\
    Z^m(H^{\bullet}(K/\mathcal{I}))  
    & i=m
      \\
      H^i(K^{\bullet} \otimes_{\mathcal{O}_T} \mathcal{O}_T/\mathcal{I}) \otimes_{\mathcal{O}_T} \mathcal{I}^{\otimes m} 
      & i \leq m-1.
    \end{cases}  
\]

Moreover, the connecting morphism
\[
\eta_{\mathcal{I},m} K^{\bullet}/\eta_{\mathcal{I},m+1} K^{\bullet}
\longrightarrow
\eta_{\mathcal{I},m+1} K^{\bullet}/
(\eta_{\mathcal{I},m} K^{\bullet} \otimes_{\mathcal{O}_T} \mathcal{I} )[1]
\]
of the distinguished triangle
\[
\eta_{\mathcal{I},m+1} K^{\bullet}/
(\eta_{\mathcal{I},m} K^{\bullet} \otimes_{\mathcal{O}_T} \mathcal{I} )
\longrightarrow
\eta_{\mathcal{I},m} K^{\bullet}/
(\eta_{\mathcal{I},m} K^{\bullet} \otimes_{\mathcal{O}_T} \mathcal{I} )
\longrightarrow
\eta_{\mathcal{I},m} K^{\bullet}/\eta_{\mathcal{I},m+1} K^{\bullet}
\longrightarrow [1]
\]
factorizes as
\[
\begin{tikzcd}
\eta_{\mathcal{I},m} K^{\bullet}/\eta_{\mathcal{I},m+1} K^{\bullet}
\arrow[r]
\arrow[d]
&
(\eta_{\mathcal{I},m+1} K^{\bullet}/
(\eta_{\mathcal{I},m} K^{\bullet} \otimes_{\mathcal{O}_T} \mathcal{I} ))[1]
\\
 H^{\bullet}(K/\mathcal{I})^m [-m] 
 \arrow[r,"\beta"]
&
Z^{m+1}(H^{\bullet}(K/\mathcal{I})) [-m]
\arrow[u]
\end{tikzcd}
\]
where $\beta$ is the differential of the complex $H^{\bullet}(K/\mathcal{I})$, namely the Bockstein map. 

Lastly, the long exact sequence associated to the distinguished triangle gives 
\[
0 \longrightarrow 
H^m(\eta_{\mathcal{I},m} K^{\bullet} \otimes_{\mathcal{O}_T} \mathcal{O}_T /\mathcal{I})
\longrightarrow
H^m(\eta_{\mathcal{I},m} K^{\bullet}/\eta_{\mathcal{I},m+1} K^{\bullet})
\]
\[
\longrightarrow
H^{m+1}( \eta_{\mathcal{I},m+1} K^{\bullet}/
(\eta_{\mathcal{I},m} K^{\bullet} \otimes_{\mathcal{O}_T} \mathcal{I} ))
\longrightarrow 
H^{m+1}( \eta_{\mathcal{I},m} K^{\bullet} \otimes_{\mathcal{O}_T} \mathcal{O}_T /\mathcal{I}) 
\longrightarrow 0,
\]
which is identified with
\[
0 \longrightarrow 
Z^m(H^{\bullet}(K/\mathcal{I}))
\longrightarrow
H^{\bullet}(K/\mathcal{I})^m
\overset{\beta}{\longrightarrow}
Z^{m+1}(H^{\bullet}(K/\mathcal{I}))
\longrightarrow 
H^{m+1}( H^{\bullet}(K/\mathcal{I})) 
\longrightarrow 0.
\]
\end{corollary}

\begin{proof}
The statement follows immediately from the long exact sequence associated to the distinguished triangle together with proposition \ref{modI} and lemma \ref{graded}, once we have identified the connecting morphism. 

The factorization of the connecting morphism follows from an easy t-structure argument, by noting that the source (resp. target) lies in $D_{\leq m}(\mathcal{O}_T)$ (resp. $D_{\geq m}(\mathcal{O}_T)$).  We now prove that the factorization $\beta$ is the Bockstein differential map. This follows easily from a diagram chasing. We write down the diagram of the distinguished triangle
\[
\begin{tikzcd}
\vdots
\arrow[r]
&
\vdots
\arrow[r]
& 
0
\\
(\eta_{\mathcal{I}} K)^{m+2}
\otimes_{\mathcal{O}_T} \mathcal{O}_T/\mathcal{I}
\arrow[r,equal]
\arrow[u]
& 
(\eta_{\mathcal{I}} K)^{m+2}
\otimes_{\mathcal{O}_T} \mathcal{O}_T/\mathcal{I}
\arrow[u]
\arrow[r]
&
0 
\arrow[u]
\\
(\eta_{\mathcal{I}} K)^{m+1}
\otimes_{\mathcal{O}_T} \mathcal{O}_T/\mathcal{I}
\arrow[u]
\arrow[r,equal]
&
(\eta_{\mathcal{I}} K)^{m+1}
\otimes_{\mathcal{O}_T} \mathcal{O}_T/\mathcal{I}
\arrow[u]
\arrow[r]
&
0
\arrow[u]
\\
\frac{K^m }{ \{ x \in K^m | dx \in \mathcal{I} K^{m+1} \}  } 
\otimes_{\mathcal{O}_T} \mathcal{I}^{\otimes m+1}
\arrow[u]
\arrow[r, ]
&
(\eta_{\mathcal{I}} K)^{m}
\otimes_{\mathcal{O}_T} \mathcal{O}_T/\mathcal{I}
\arrow[u]
\arrow[r]
&
Z^m(K^{\bullet}/\mathcal{I}K^{\bullet})
\otimes_{\mathcal{O}_T}  \mathcal{I}^{\otimes m}
\arrow[u]
\\
0
\arrow[r]
\arrow[u]
&
 (K^{m-1}/\mathcal{I}K^{m-1} )\otimes_{\mathcal{O}_T} \mathcal{I}^{\otimes m}
\arrow[u]
\arrow[r,equal]
&
 (K^{m-1}/\mathcal{I}K^{m-1} )\otimes_{\mathcal{O}_T} \mathcal{I}^{\otimes m}.
 \arrow[u]
\end{tikzcd}
\]
Given $\bar{x} \in Z^m(K^{\bullet}/\mathcal{I}K^{\bullet})$
representing an element of $H^{\bullet}(K/\mathcal{I})^m$, we choose a lift $x \in \{ x \in K^m | dx \in \mathcal{I} K^{m+1} \}$ 
of it. Then $\beta(\bar{x})$ is represented by $dx$ 
viewed as an element of
$ \mathcal{I} K^{m+1} / \mathcal{I} \{ x \in K^{m+1} | dx \in \mathcal{I} K^{m+2} \}$, but this clearly also represents the Bockstein map of $\bar{x}$. 
\end{proof}

We have another natural inclusions
\[
\eta_{\mathcal{I},m+1} K^{\bullet} \otimes_{\mathcal{O}_T} \mathcal{I}
\subset
\eta_{\mathcal{I},m} K^{\bullet} \otimes_{\mathcal{O}_T} \mathcal{I}  \subset \eta_{\mathcal{I},m+1} K^{\bullet} 
\]
giving rise to the distinguished triangle
\[
(\eta_{\mathcal{I},m} K^{\bullet}/\eta_{\mathcal{I},m+1} K^{\bullet}) (1)
\longrightarrow
\eta_{\mathcal{I},m+1} K^{\bullet}/
(\eta_{\mathcal{I},m+1} K^{\bullet} \otimes_{\mathcal{O}_T} \mathcal{I} )
\longrightarrow
\eta_{\mathcal{I},m+1} K^{\bullet}/
(\eta_{\mathcal{I},m} K^{\bullet} \otimes_{\mathcal{O}_T} \mathcal{I} )
\longrightarrow [1]
\]
with connecting morphism
\[
F_{m+1}H^{\bullet}(K/\mathcal{I})
\cong 
\eta_{\mathcal{I},m+1} K^{\bullet}/
(\eta_{\mathcal{I},m} K^{\bullet} \otimes_{\mathcal{O}_T} \mathcal{I} )
\longrightarrow 
(\eta_{\mathcal{I},m} K^{\bullet}/\eta_{\mathcal{I},m+1} K^{\bullet}) (1)[1]
\cong 
\tau_{\leq m} (K^{\bullet} /\mathcal{I}) (m+1) [1],
\]
which has to be zero since the source and target sit in different cohomological degrees. Thus we have a splitting
\[
\eta_{\mathcal{I},m+1} K^{\bullet}/
(\eta_{\mathcal{I},m+1} K^{\bullet} \otimes_{\mathcal{O}_T} \mathcal{I} ) 
\cong 
(\eta_{\mathcal{I},m} K^{\bullet}/\eta_{\mathcal{I},m+1} K^{\bullet}) (1)
\oplus 
\eta_{\mathcal{I},m+1} K^{\bullet}/
(\eta_{\mathcal{I},m} K^{\bullet} \otimes_{\mathcal{O}_T} \mathcal{I} ).
\]
Then proposition \ref{modI} and lemma \ref{graded} identifies the direct summand

\begin{corollary} \label{jdiiii}
Let $K^{\bullet}$ be a $\mathcal{I}$-torsion-free complex, we have an identification
\[
\eta_{\mathcal{I},m+1} K^{\bullet}/
(\eta_{\mathcal{I},m+1} K^{\bullet} \otimes_{\mathcal{O}_T} \mathcal{I} ) 
\cong 
(\tau_{\leq m} K^{\bullet} \otimes_{\mathcal{O}_T} \mathcal{O}_T/   \mathcal{I}) (m+1)
\oplus
F_{m+1} H^{\bullet} (K^{\bullet}/\mathcal{I}),
\]
which splits the distinguished triangle
\[
(\tau_{\leq m} K^{\bullet} \otimes_{\mathcal{O}_T} \mathcal{O}_T/   \mathcal{I}) (m+1)
\longrightarrow
\eta_{\mathcal{I},m+1} K^{\bullet}/
(\eta_{\mathcal{I},m+1} K^{\bullet} \otimes_{\mathcal{O}_T} \mathcal{I} )
\longrightarrow
F_{m+1} H^{\bullet} (K^{\bullet}/\mathcal{I})
\longrightarrow [1].
\]

Moreover,  the distinguished triangle
\[
\eta_{\mathcal{I},m+1} K^{\bullet}/
(\eta_{\mathcal{I},m} K^{\bullet} \otimes_{\mathcal{O}_T} \mathcal{I} )
\longrightarrow
\eta_{\mathcal{I},m} K^{\bullet}/
(\eta_{\mathcal{I},m} K^{\bullet} \otimes_{\mathcal{O}_T} \mathcal{I} )
\longrightarrow
\eta_{\mathcal{I},m} K^{\bullet}/\eta_{\mathcal{I},m+1} K^{\bullet}
\longrightarrow [1]
\]
is compatible with the splitting in the sense that we have a commutative diagram
\[
\begin{tikzcd}
& (\tau_{\leq m-1} K^{\bullet} \otimes_{\mathcal{O}_T} \mathcal{O}_T/   \mathcal{I}) (m)
\arrow[d]
\arrow[rd,"b"]
& &
\\
F_{m+1} H^{\bullet} (K^{\bullet}/\mathcal{I})
\arrow[r]
\arrow[rd,"a"]
&
\eta_{\mathcal{I},m} K^{\bullet}/
(\eta_{\mathcal{I},m} K^{\bullet} \otimes_{\mathcal{O}_T} \mathcal{I} )
\arrow[r]
\arrow[d]
&
(\tau_{\leq m} K^{\bullet} \otimes_{\mathcal{O}_T} \mathcal{O}_T/   \mathcal{I}) (m)
\arrow[r]
\arrow[d,"d"]
&
\text{$[1]$}
\\
&
F_{m} H^{\bullet} (K^{\bullet}/\mathcal{I})
\arrow[d]
\arrow[r,"c"]
&
H^m(K^{\bullet} \otimes_{\mathcal{O}_T} \mathcal{O}_T/   \mathcal{I}) (m)[-m]
&
\\
& \text{$[1]$} & & 
\end{tikzcd}
\]
where the the arrows $a$ and $b$ are the canonical map corresponding to the Hodge filtration and standard truncated filtration respectively, while $c$ and $d$ are graded quotient maps corresponding to them. 
\end{corollary}

\section{Proof}

We now start the proof of the theorem. We first observe that the Bialynicki-Birula type construction of Fil$_m$ is fundamentally on cohomology groups, whereas the Hodge-Tate filtration, in the classical construction, originates from a filtration on complexes. Thus a natural way to proceed is to upgrade $\text{Fil}_m$ to a filtration on complexes and then compare the two filtration on the derived category level. This is achieved by the $L\eta_{\xi,\bullet}$-operation introduced in the previous section. In some sense, we need to have a suitably derived Bialynicki-Birula
construction. 

Let $m \in \mathbb{Z}_{\geq 0}$, and we consider 
\[
L\eta_{\xi,m} \prism_{X/\mathbb{B}^+_{\text{dR}}}
\]
which is naturally a subcomplex of $L\eta_{\xi}\prism_{X/\mathbb{B}^+_{\text{dR}}}$. 
By lemma \ref{graded}, we have
\[
Gr^m  \prism_{X/\mathbb{B}^+_{\text{dR}}}:= 
L\eta_{\xi,m} \prism_{X/\mathbb{B}^+_{\text{dR}}}/
L\eta_{\xi,m+1} \prism_{X/\mathbb{B}^+_{\text{dR}}}
\footnote{The notation is most naturally interpreted as the cokernel in the stable infinity category enrichment of the derived category. Alternatively, we can interpret it as the $m$-th graded quotient of $L\eta_{\xi, \bullet} \prism_{X/\mathbb{B}^+_{\text{dR}}}$
in the filtered derived category. }
\cong 
\tau_{\leq m}
\overline{ \prism}_{X/\mathbb{B}^+_{\text{dR}}} (m)
\cong 
\tau_{\leq m} R\nu_* \hat{\mathcal{O}}_X (m),
\]
which implies that under the canonical map 
\[
L\eta_{\xi,m} \prism_{X/\mathbb{B}^+_{\text{dR}}} 
\longrightarrow 
\xi^m \prism_{X/\mathbb{B}^+_{\text{dR}}}
\longrightarrow 
\overline{ \prism}_{X/\mathbb{B}^+_{\text{dR}}}(m), 
\]
the image of
$H^i(X, L\eta_{\xi,m} \prism_{X/\mathbb{B}^+_{\text{dR}}})$
inside 
\[
H^i(X,\overline{ \prism}_{X/\mathbb{B}^+_{\text{dR}}})(m)
\cong 
H^{i}_{\text{ét}}(X,\mathbb{Q}_p) \otimes_{\mathbb{Q}_p} \mathbb{C} (m)
\]
is exactly ($m$-th Tate twist of) the $m$-th Hodge-Tate filtration. 

We observe that there is a natural map
\[
H^i(X, L\eta_{\xi,m} \prism_{X/\mathbb{B}^+_{\text{dR}}})
\longrightarrow
H^i(X, L\eta_{\xi} \prism_{X/\mathbb{B}^+_{\text{dR}}})
\cap
\xi^mH^i(X, \prism_{X/\mathbb{B}^+_{\text{dR}}} ), 
\]
and we want to prove that it induces an identification of
$H^i(X, L\eta_{\xi,m} \prism_{X/\mathbb{B}^+_{\text{dR}}})$
with 
\[
H^i(X, L\eta_{\xi} \prism_{X/\mathbb{B}^+_{\text{dR}}})
\cap
\xi^mH^i(X, \prism_{X/\mathbb{B}^+_{\text{dR}}} )
\cong
H^i_{\text{crys}}(X/\mathbb{B}^+_{\text{dR}}) \cap 
\xi^m
(H^{i}_{\text{ét}}(X,\mathbb{Q}_p) \otimes_{\mathbb{Q}_p} \mathbb{B}_{\text{dR}}^+)), 
\]
which finishes the proof.

\begin{proposition} \label{torsionfree}
$H^i(X, L\eta_{\xi,m} \prism_{X/\mathbb{B}^+_{\text{dR}}})$ 
is $\xi$-torsion-free for all $i, m \in \mathbb{Z}_{\geq 0}$. 
\end{proposition}

\begin{myproof}[First Proof]
We prove that 
$H^i(X, L\eta_{\xi,m} \prism_{X/\mathbb{B}^+_{\text{dR}}})$ 
is $\xi$-torsion-free by descending induction on $m$. The base case is when $m$ is big enough so we have 
$L\eta_{\xi,m} \prism_{X/\mathbb{B}^+_{\text{dR}}} = \prism_{X/\mathbb{B}^+_{\text{dR}}}(m)$, in which case the torsion-freeness follows from the primitive comparison theorem
\[
H^i(X, L\eta_{\xi,m} \prism_{X/\mathbb{B}^+_{\text{dR}}}) \cong 
H^i(X,\prism_{X/\mathbb{B}^+_{\text{dR}}}(m))
\cong
H^i_{\text{ét}}(X, \mathbb{Q}_p) \otimes_{\mathbb{Q}_p} \mathbb{B}^+_{\text{dR}}(m). 
\]

Now assume that 
$H^i(X, L\eta_{\xi,m+1} \prism_{X/\mathbb{B}^+_{\text{dR}}})$ 
is $\xi$-torsion-free, we want to prove that the same is true for 
$H^i(X, L\eta_{\xi,m} \prism_{X/\mathbb{B}^+_{\text{dR}}})$. 
Using the long exact sequence associated to
\[
 L\eta_{\xi,m} \prism_{X/\mathbb{B}^+_{\text{dR}}} \overset{\xi}{\longrightarrow }
L\eta_{\xi,m} \prism_{X/\mathbb{B}^+_{\text{dR}}}
\longrightarrow 
L\eta_{\xi,m} \prism_{X/\mathbb{B}^+_{\text{dR}}}/ \xi L\eta_{\xi,m} \prism_{X/\mathbb{B}^+_{\text{dR}}}
\longrightarrow [1],
\]
it is enough to prove that the natural map 
\[
H^i(X, L\eta_{\xi,m} \prism_{X/\mathbb{B}^+_{\text{dR}}})
\longrightarrow 
H^i(X, L\eta_{\xi,m} \prism_{X/\mathbb{B}^+_{\text{dR}}}/ \xi L\eta_{\xi,m} \prism_{X/\mathbb{B}^+_{\text{dR}}})
\]
is surjective.  

We have the commutative diagram
\[
\begin{tikzcd}
H^i(X,L\eta_{\xi,m+1} \prism_{X/\mathbb{B}^+_{\text{dR}}}/\xi L\eta_{\xi,m} \prism_{X/\mathbb{B}^+_{\text{dR}}})
\arrow[r,"b"]
&
H^i(X, L\eta_{\xi,m} \prism_{X/\mathbb{B}^+_{\text{dR}}}/\xi L\eta_{\xi,m} \prism_{X/\mathbb{B}^+_{\text{dR}}}) 
\arrow[r,"d"]
& 
H^{i} (X, Gr^m  \prism_{X/\mathbb{B}^+_{\text{dR}}})
\\
H^i(X, L\eta_{\xi,m+1} \prism_{X/\mathbb{B}^+_{\text{dR}}})
\arrow[u,"h"]
\arrow[r,"g"]
&
H^i(X, L\eta_{\xi,m} \prism_{X/\mathbb{B}^+_{\text{dR}}}) 
\arrow[r,twoheadrightarrow,"a"]
\arrow[u,"f"]
&
H^{i} (X, Gr^m  \prism_{X/\mathbb{B}^+_{\text{dR}}})
\arrow[u,equal]
\end{tikzcd}
\]
corresponding to the map of distinguished triangles
\[
\begin{tikzcd}
L\eta_{\xi,m+1} \prism_{X/\mathbb{B}^+_{\text{dR}}}/\xi L\eta_{\xi,m} \prism_{X/\mathbb{B}^+_{\text{dR}}}
\arrow[r]
&
 L\eta_{\xi,m} \prism_{X/\mathbb{B}^+_{\text{dR}}}/\xi L\eta_{\xi,m} \prism_{X/\mathbb{B}^+_{\text{dR}}} 
\arrow[r]
& 
 Gr^m  \prism_{X/\mathbb{B}^+_{\text{dR}}}
\\
 L\eta_{\xi,m+1} \prism_{X/\mathbb{B}^+_{\text{dR}}}
\arrow[u,]
\arrow[r,]
&
 L\eta_{\xi,m} \prism_{X/\mathbb{B}^+_{\text{dR}}}
\arrow[r,twoheadrightarrow,]
\arrow[u]
&
 Gr^m  \prism_{X/\mathbb{B}^+_{\text{dR}}}. 
\arrow[u,equal]
\end{tikzcd}
\]
We note that the map $a$ is surjective by our induction hypothesis. Namely, the connecting map \[
H^{i} (X, Gr^m  \prism_{X/\mathbb{B}^+_{\text{dR}}})
\longrightarrow
H^{i+1}(X, L\eta_{\xi,m+1} \prism_{X/\mathbb{B}^+_{\text{dR}}})
\]
is 0 since the right hand side is torsion free by induction hypothesis. 

We have that $h$ is surjective as well. This is again because of our inductive hypothesis that
$H^{i}(X, L\eta_{\xi,m+1} \prism_{X/\mathbb{B}^+_{\text{dR}}})$
is torsion-free, so the natural map
\[
p: 
H^{i}(X, L\eta_{\xi,m+1} \prism_{X/\mathbb{B}^+_{\text{dR}}})
\longrightarrow
H^{i}(X, L\eta_{\xi,m+1} \prism_{X/\mathbb{B}^+_{\text{dR}}}/\xi L\eta_{\xi,m+1} \prism_{X/\mathbb{B}^+_{\text{dR}}})
\]
is surjective, but corollary \ref{jdiiii} implies that 
\[
H^{i}(X, L\eta_{\xi,m+1} \prism_{X/\mathbb{B}^+_{\text{dR}}}/\xi L\eta_{\xi,m+1} \prism_{X/\mathbb{B}^+_{\text{dR}}})
\cong 
H^i(X, (Gr^m \prism_{X/\mathbb{B}^+_{\text{dR}}})) (1)
\oplus 
H^i(X,
L\eta_{\mathcal{I},m+1} \prism_{X/\mathbb{B}^+_{\text{dR}}}/\xi L\eta_{\mathcal{I},m} \prism_{X/\mathbb{B}^+_{\text{dR}}})
\]
composing the projection to second factor with $p$ implies that $h$
is surjective. 

Now we prove that $f$ is surjective, which finishes the proof. Given
$x \in H^i(X, L\eta_{\xi,m} \prism_{X/\mathbb{B}^+_{\text{dR}}}/\xi L\eta_{\xi,m} \prism_{X/\mathbb{B}^+_{\text{dR}}})$, 
we can find 
$y \in H^i(X, L\eta_{\xi,m} \prism_{X/\mathbb{B}^+_{\text{dR}}})$
such that 
$d(f(y)) = a(y) = d(x)$
by surjectivity of $a$. Thus
$x-f(y) \in \text{Ker}(d)=\text{Im}(b)$, so there exists
$w \in H^i(X,L\eta_{\xi,m+1} \prism_{X/\mathbb{B}^+_{\text{dR}}}/\xi L\eta_{\xi,m} \prism_{X/\mathbb{B}^+_{\text{dR}}})$
such that 
$b(w) =x-f(y)$, 
the surjectivity of $h$ tells us that 
$w= h(z)$
for $z \in H^i(X,L\eta_{\xi,m+1} \prism_{X/\mathbb{B}^+_{\text{dR}}})$, which implies that 
\[
x= f(y)+b(h(z)) = f(y + g(z)),
\]
proving the surjectivity of $f$.
\end{myproof}

\begin{myproof}[Second Proof]
We prove that 
$H^i(X, L\eta_{\xi,m} \prism_{X/\mathbb{B}^+_{\text{dR}}})$ 
is $\xi$-torsion-free by ascending induction on $m$. The base case $m=0$ is
\cite{Bhatt2018} theorem 13.19 since 
\[
L\eta_{\xi,0} \prism_{X/\mathbb{B}^+_{\text{dR}}} 
= L\eta_{\xi} \prism_{X/\mathbb{B}^+_{\text{dR}}},
\]
which follows from 
$\prism_{X/\mathbb{B}^+_{\text{dR}}} \in D_{\geq 0}$ and
$\xi$-torsion-freeness of $H^0(\prism_{X/\mathbb{B}^+_{\text{dR}}})$. 

Now assume that 
$H^i(X, L\eta_{\xi,m} \prism_{X/\mathbb{B}^+_{\text{dR}}})$ 
is $\xi$-torsion-free, we want to prove that the same is true for 
$H^i(X, L\eta_{\xi,m+1} \prism_{X/\mathbb{B}^+_{\text{dR}}})$. 
It is enough to prove that the natural map 
\[
H^i(X, L\eta_{\xi,m+1} \prism_{X/\mathbb{B}^+_{\text{dR}}})
\longrightarrow 
H^i(X, L\eta_{\xi,m} \prism_{X/\mathbb{B}^+_{\text{dR}}})
\]
is injective. Using the long exact sequence associated to the distinguished triangle
\begin{equation} \label{triangle}
L\eta_{\xi,m+1} \prism_{X/\mathbb{B}^+_{\text{dR}}} \longrightarrow 
L\eta_{\xi,m} \prism_{X/\mathbb{B}^+_{\text{dR}}}
\longrightarrow 
Gr^m  \prism_{X/\mathbb{B}^+_{\text{dR}}}
\longrightarrow [1],
\end{equation}
it is enough to show that 
\[
H^{i-1} (X, Gr^m  \prism_{X/\mathbb{B}^+_{\text{dR}}})
\longrightarrow 
H^i(X, L\eta_{\xi,m+1} \prism_{X/\mathbb{B}^+_{\text{dR}}})
\]
is 0. Note that the $i=0$ case is trivial, as the left hand side cohomology group is 0. 

We have the commutative diagram
\[
\begin{tikzcd}
H^{i-1} (X, Gr^m  \prism_{X/\mathbb{B}^+_{\text{dR}}})
\arrow[r,"\beta"]
&
H^i(X,L\eta_{\xi,m+1} \prism_{X/\mathbb{B}^+_{\text{dR}}}/\xi L\eta_{\xi,m} \prism_{X/\mathbb{B}^+_{\text{dR}}})
& 
\\
H^{i-1} (X, Gr^m  \prism_{X/\mathbb{B}^+_{\text{dR}}})
\arrow[r,"f"]
\arrow[u,equal]
& 
H^i(X, L\eta_{\xi,m+1} \prism_{X/\mathbb{B}^+_{\text{dR}}})
\arrow[u,"h"]
\arrow[r,"g"]
&
H^i(X, L\eta_{\xi,m} \prism_{X/\mathbb{B}^+_{\text{dR}}}) 
\\
&
H^i(X, \xi L\eta_{\xi,m} \prism_{X/\mathbb{B}^+_{\text{dR}}})
\arrow[ur, hook, "b"]
\arrow[u,"a"]
&
\end{tikzcd}
\]
where the middle row and column are exact, being part of the long exact sequence associated to  distinguished triangles. The first row is the connecting morphism as specified in corollary \ref{corollary}, and the square is commutative since it is induced by the morphism of distinguished triangles
\[
\begin{tikzcd}
L\eta_{\xi,m+1} \prism_{X/\mathbb{B}^+_{\text{dR}}}/\xi L\eta_{\xi,m} \prism_{X/\mathbb{B}^+_{\text{dR}}}
\arrow[r,""]
&
L\eta_{\xi,m} \prism_{X/\mathbb{B}^+_{\text{dR}}}/\xi L\eta_{\xi,m} \prism_{X/\mathbb{B}^+_{\text{dR}}}
\arrow[r]
&
Gr^m  \prism_{X/\mathbb{B}^+_{\text{dR}}}
\arrow[r]
&
\text{[1]}
\\
L\eta_{\xi,m+1} \prism_{X/\mathbb{B}^+_{\text{dR}}}
\arrow[r]
\arrow[u]
& 
 L\eta_{\xi,m} \prism_{X/\mathbb{B}^+_{\text{dR}}}
\arrow[u]
\arrow[r]
&
Gr^m  \prism_{X/\mathbb{B}^+_{\text{dR}}} 
\arrow[u] 
\arrow[r]
&
\text{[1]}.
\arrow[u]
\end{tikzcd}
\]
The arrow $b$ is injective by our inductive hypothesis, so $a$ is injective as well.  

We know from (\ref{derham}) and corollary \ref{corollary} that $\beta$ factorizes as 
\[
\begin{tikzcd}
H^{i-1} (X, Gr^m  \prism_{X/\mathbb{B}^+_{\text{dR}}})
\arrow[d,equal, "{\resizebox{0.8cm}{0.1cm}{$\sim$}}" labl]
\arrow[r,"\beta"] 
&
H^i(X,L\eta_{\xi,m+1} \prism_{X/\mathbb{B}^+_{\text{dR}}}/\xi L\eta_{\xi,m} \prism_{X/\mathbb{B}^+_{\text{dR}}})
\arrow[d,equal, "{\resizebox{0.8cm}{0.1cm}{$\sim$}}" labl]
\\
H^{i-1} (X, \tau_{\leq m}\overline{  \prism}_{X/\mathbb{B}^+_{\text{dR}}}) (m)
\arrow[r,"\beta"]
\arrow[d]
&
H^i(X,F_{m+1} \Omega_X^{\bullet})
\\
 H^{i-1}(X, \Omega_X^m[-m])
 \arrow[r,"d"]
&
H^{i-1}(X, Z^{m+1}\Omega_X^{\bullet}[-m])
\arrow[u]
\end{tikzcd}
\]
where $Z^{m+1}\Omega_X^{\bullet}$ is the sheaf of closed $m+1$-forms, and $d$ is the usual differential of de Rham complexes.  By the degeneration of Hodge-de Rham spectral sequence proved in \cite{Bhatt2018} theorem 13.3, we know that 
\[
H^{i-1}(X, \Omega_X^m) \overset{d}{\longrightarrow} 
H^{i-1}(X, \Omega_X^{m+1})
\]
is $0$, and we claim that this implies  that the composition
\[
H^{i-1}(X, \Omega_X^m[-m]) \overset{d}{\longrightarrow} 
H^{i-1}(X, Z^{m+1}\Omega_X^{\bullet}[-m]) 
\longrightarrow 
H^i(X,F_{m+1} \Omega_X^{\bullet}) \cong 
H^{i-1}(X,F_{m+1} \Omega_X^{\bullet}[1])
\]
is 0, so $\beta$ is 0 as well. 

Indeed, filtering both $ \Omega_X^m[-m]$ and $F_{m+1} \Omega_X^{\bullet}[1]$
by the Hodge filtration, we see that the map on the  first graded piece is 
$H^{i-1}(X, \Omega_X^m[-m]) \overset{d}{\longrightarrow} 
H^{i-1}(X, \Omega_X^{m+1}[-m])$
which we have seen to be 0. The map on other graded pieces are also 0, since $\Omega_X^m[-m]$ has only one non-zero graded piece, proving the claim. 

Now let $x \in H^{i-1} (X, Gr^m  \prism_{X/\mathbb{B}^+_{\text{dR}}})$, then 
$f(x) \in \text{Ker}(h)$ since $\beta =0$, so we have $f(x) = a(y)$ for some
$y \in H^i(X, \xi L\eta_{\xi,m} \prism_{X/\mathbb{B}^+_{\text{dR}}})$. Then $b(y)=g(a(y))=g(f(x))=0$, which implies that $y=0$ by injectivity of $b$, so $f(x)=a(y)=0$. Hence we have $f=0$, finishing the induction. 
\end{myproof}

\begin{remark}
The first proof has the advantage that it is independent of the Conrad-Gabber spreading theorem. In particular, we give a Conrad-Gabber independent proof the $\xi$-torsion-freeness of 
$H^i_{crys}(X,\mathbb{B}^+_{\text{dR}})$. 
See  \cite{Bhatt2018} 13.19 and \cite{Guo1} theorem 7.3.5 for proofs that involves the spreading theorem. 
\end{remark}

The proposition tells us that long exact sequence assoicated to the distinguished triangle (\ref{triangle}) splits into short exact sequences
\[
0 \longrightarrow
H^i(X, L\eta_{\xi,m+1} \prism_{X/\mathbb{B}^+_{\text{dR}}}) 
\longrightarrow
H^i(X, L\eta_{\xi,m} \prism_{X/\mathbb{B}^+_{\text{dR}}})
\longrightarrow
H^i(X, \tau_{\leq m}\overline{  \prism}_{X/\mathbb{B}^+_{\text{dR}}}) (m)
\longrightarrow 0,
\]
indeed, 
$H^i(X, \tau_{\leq m}\overline{  \prism}_{X/\mathbb{B}^+_{\text{dR}}}) (m)$
is $\xi$-torsion which admits no non-zero connecting morphism to the $\xi$-torsion-free module 
$H^{i+1}(X, L\eta_{\xi,m+1} \prism_{X/\mathbb{B}^+_{\text{dR}}}) $. 
In particular, 
\[
H^i(X, L\eta_{\xi,m} \prism_{X/\mathbb{B}^+_{\text{dR}}}) \hookrightarrow 
H^i(X, L\eta_{\xi,0} \prism_{X/\mathbb{B}^+_{\text{dR}}}) =
H^i(X, L\eta_{\xi} \prism_{X/\mathbb{B}^+_{\text{dR}}}),
\]
so the canonical map
\[
H^i(X, L\eta_{\xi,m} \prism_{X/\mathbb{B}^+_{\text{dR}}})
\longrightarrow
H^i(X, L\eta_{\xi} \prism_{X/\mathbb{B}^+_{\text{dR}}})
\cap
\xi^mH^i(X, \prism_{X/\mathbb{B}^+_{\text{dR}}} )
\]
is injective. 
We now prove that it is also surjective.

We proceed by induction on $m$. The base case $m=0$ is clear, as both sides are 
$H^i(X, L\eta_{\xi} \prism_{X/\mathbb{B}^+_{\text{dR}}})$. Assume that the map is surjective for $m$, we prove it is also surjective for $m+1$. We consider the commutative diagram 
\[
\begin{tikzcd}
0
\arrow[d]
&
0
\arrow[d]
\\
H^i(X, L\eta_{\xi,m+1} \prism_{X/\mathbb{B}^+_{\text{dR}}}) 
\arrow[r]
\arrow[d]
& 
H^i(X, L\eta_{\xi} \prism_{X/\mathbb{B}^+_{\text{dR}}})
\cap
\xi^{m+1}H^i(X, \prism_{X/\mathbb{B}^+_{\text{dR}}} )
\arrow[d]
\\
H^i(X, L\eta_{\xi,m} \prism_{X/\mathbb{B}^+_{\text{dR}}})
\arrow[r,"\sim"]
\arrow[d]
&
H^i(X, L\eta_{\xi} \prism_{X/\mathbb{B}^+_{\text{dR}}})
\cap
\xi^mH^i(X, \prism_{X/\mathbb{B}^+_{\text{dR}}} )
\arrow[d]
\\
H^i(X, \tau_{\leq m}\overline{  \prism}_{X/\mathbb{B}^+_{\text{dR}}}) (m)
\arrow[r,hook]
\arrow[d]
&
H^i(X, \overline{  \prism}_{X/\mathbb{B}^+_{\text{dR}}}) (m)
\\
0
&
\end{tikzcd}
\]
where the two vertical sequences are exact, the middle horizontal arrow is an isomorphism by our inductive hypothesis, and the bottom horizontal arrow is injective by the degeneration of Hodge-Tate spectral sequence (\cite{Bhatt2018} theorem 13.3). Now an easy diagram chasing 
proves that the first horizontal arrow is surjective, finishing the induction. 

\section{Miscellany}

We document an interesting byproduct in our treatment of Hodge-Tate filtration, namely we can give a conceptual explanation why the degeneration of Hodge-Tate spectral sequences is equivalent to that of Hodge-de Rham spectral sequences. The claim is clear  in the proof of  \cite{Bhatt2018} theorem 13.3, which is by dimension counting. Note that the torsion-freeness of $H^i_{\text{crys}}(X/\mathbb{B}^+_{\text{dR}})$
is used in $loc.cit.$ as a bridge between the dimension of the de Rham cohomology and étale cohomology. 

\begin{proposition}

Let $X$ be a proper smooth rigid analytic variety over a complete  algebraically closed non-archimedean field $\mathbb{C}$
of mixed characteristic $p$, the degeneration of the Hodge-Tate spectral sequence
\[
E_2^{p,q} = H^{p}(X,\Omega_{X}^q)(-q) \ \ \Longrightarrow H^{p+q}_{\text{ét}}(X,\mathbb{Q}_p) \otimes_{\mathbb{Q}_p} \mathbb{C}
\]
is equivalent to the degeneration of Hodge-de Rham spectral sequence 
\[
E_1^{p,q}=H^{q}(X,\Omega_{X}^p) \Longrightarrow H^{p+q}(X,\Omega_X^{\bullet}). 
\]
More precisely, in the natural diagram
\[
\begin{tikzcd}
H^i(X, \tau_{\leq m-1}\overline{\prism}_{X/\mathbb{B}^+_{\text{dR}}})(m) 
\arrow[r,"f"]
&
H^i(X, \tau_{\leq m}\overline{\prism}_{X/\mathbb{B}^+_{\text{dR}}})(m)
\arrow[r,]
&
H^{i}(X,\Omega_{X}^m[-m])
\\
H^{i}(X,F_{m+1}\Omega_{X}^{\bullet})
\arrow[r,"g"]
&
H^{i}(X,F_m\Omega_{X}^{\bullet})
\arrow[r]
&
H^{i}(X,\Omega_{X}^m[-m])
\arrow[u,equal]
\end{tikzcd}
\]
where $f$ and $g$ are the canonical maps 
 corresponding to the truncation filtration and Hodge filtration respectively, we have that
 \[
 Coker(f) = Coker(g) \subset H^{i}(X,\Omega_{X}^m[-m])
 \]
 as subspaces of $H^{i}(X,\Omega_{X}^m[-m])$. 
\end{proposition}

\begin{proof}
We know that the Hodge-Tate spectral sequence is induced from the filtration
$H^i(X, \tau_{\leq m}\overline{\prism}_{X/\mathbb{B}^+_{\text{dR}}})$
corresponding to standard truncation, and the Hodge-de Rham spectral sequence is induced from the Hodge filtration 
$H^{i}(X,F_m\Omega_{X}^{\bullet})$. It is then clear that the equivalence of degeneration is implied by the claim on $f$ and $g$. 

We now prove the claim.  We have a commutative diagram
\[
\begin{tikzcd}
& H^i(X, L\eta_{\xi,m} \prism_{X/\mathbb{B}^+_{\text{dR}}})
\arrow[d, hook, "\xi"] 
& &
\\
H^i(X, L\eta_{\xi,m+1} \prism_{X/\mathbb{B}^+_{\text{dR}}}) 
\arrow[r,hook]
&
H^i(X, L\eta_{\xi,m} \prism_{X/\mathbb{B}^+_{\text{dR}}})
\arrow[r,twoheadrightarrow]
\arrow[d,twoheadrightarrow]
&
H^i(X, \tau_{\leq m}\overline{  \prism}_{X/\mathbb{B}^+_{\text{dR}}}) (m)
& 
\\
&
H^i(X, L\eta_{\xi,m} \prism_{X/\mathbb{B}^+_{\text{dR}}}/\xi L\eta_{\xi,m} \prism_{X/\mathbb{B}^+_{\text{dR}}})
\arrow[ru,twoheadrightarrow,"h"]
& &
\end{tikzcd}
\]
where the vertical and horizontal sequences are both short exact sequences, which follows from proposition \ref{torsionfree}. The factorization $h$ is induced by the canonical map
\[
L\eta_{\xi,m} \prism_{X/\mathbb{B}^+_{\text{dR}}}/\xi L\eta_{\xi,m} \prism_{X/\mathbb{B}^+_{\text{dR}}}
\longrightarrow
L\eta_{\xi,m} \prism_{X/\mathbb{B}^+_{\text{dR}}}/
L\eta_{\xi,m+1} \prism_{X/\mathbb{B}^+_{\text{dR}}}
\cong 
\tau_{\leq m}\overline{  \prism}_{X/\mathbb{B}^+_{\text{dR}}} (m), 
\]
and $h$ is surjective because it is a factorization of a surjective map. 

Now corollary \ref{jdiiii} gives us a commutative diagram
\[
\begin{tikzcd}
&
H^{i}(X,F_{m+1}\Omega_{X}^{\bullet})
\arrow[d]
\arrow[rd,"g"]
&
\\
H^i(X, \tau_{\leq m-1}\overline{  \prism}_{X/\mathbb{B}^+_{\text{dR}}}) (m)
\arrow[r,hook]
\arrow[rd,"f"]
&
H^i(X, L\eta_{\xi,m} \prism_{X/\mathbb{B}^+_{\text{dR}}}/\xi L\eta_{\xi,m} \prism_{X/\mathbb{B}^+_{\text{dR}}})
\arrow[r,twoheadrightarrow]
\arrow[d,"h ",twoheadrightarrow]
&
H^{i}(X,F_m\Omega_{X}^{\bullet})
\arrow[d]
\\
&
H^i(X, \tau_{\leq m}\overline{  \prism}_{X/\mathbb{B}^+_{\text{dR}}}) (m)
\arrow[r]
&
H^{i}(X,\Omega_{X}^{m}[-m]),
\end{tikzcd}
\]
and the claim in the proposition follows immediately. 
\end{proof}

\bibliographystyle{alpha} 
\bibliography{ref}
\end{document}